\documentclass[12 pt]{article}
\usepackage{amsmath}
\usepackage{amssymb}
\usepackage{amsthm}

\newtheorem{Theorem}{Theorem}
\newtheorem{Lemma}{Lemma}
\newtheorem{Corollary}{Corollary}
\newtheorem{Proposition}{Proposition} 
\newtheorem{Remark}{Remark}
\begin{document}

\title{On the accuracy of the poissonisation in the infinite occupancy scheme}
\author{Mikhail Chebunin\thanks{E-mail: chebuninmikhail@gmail.com, Novosibirsk State University, Novosibirsk, Russia. The research was supported by RFBR grant 17-01-00683}}
\date{}
\maketitle

\begin{abstract}
We obtain asymptotic accuracy of the poissonisation in the infinite occupancy scheme.
All results are obtained for integer-valued random variables having a regularly varying distribution.
\end{abstract}


\section{Introduction}

We consider a model with $n$ balls and infinitely many cells ("urns") numbered $1, 2, \ldots$. 
Ball $j=1,2,\dots,n$ is randomly thrown to cell $X_j, \ {\mathbb P} (X_j=i)=p_i>0, \ \sum_{i=1}^\infty p_i=1$, independently of everything else.
Denote by  $J_i(n)=\sum_{j=1}^{n} {\mathbb I}(X_j=i)$ the total number of balls in cell $i$.
Let 
\begin{equation}\label{Rn}
R^{*}_{n,k}=\sum\limits_{i=1}^{\infty} {\mathbb I} (J_i(n)\ge k)
\end{equation}
 be the number of cells containing at least $k\geq 1$ balls,
$ R_{n,k}=R^{*}_{n,k}-R^{*}_{n,k+1}$  the number of cells with exactly  $k$ balls, and  assume $p_1\geq p_2 \geq \ldots $.

Karlin (1967) has obtained pioneering results in the study of this model. We recall here a number of his results.
It seems that he was the first who introduced the "poissonisation" procedure in this content. 
Namely, instead of fixed-size sampling he considered  samples of random size $P(n)$. Where
$\{P (t), \ t \ge 0 \}$ is a Poisson process with intensity one that does not depend on the procedure of assigning cells to balls.
 According to the well-known splitting property of Poisson flows, random processes $\{J_i(P(t))\stackrel{def}{=}P_i(t), \ t\geq 0\}$ are Poisson with intensities $p_i, \ i=1,2,\dots,$ and mutually independent for different $i$. 
 From (\ref{Rn}),
\[
R^{*}_{P(t),k}=\sum_{i=1}^{\infty} {\mathbb I}(P_i(t)\geq k) \ \ \textrm{and} \ \ R_{P(t),k}=\sum_{i=1}^{\infty} {\mathbb I}(P_i(t)= k).
\]
Let $ \alpha(x)=\max\{j\ : \ \ p_j\geq 1/x\}$ and assume the function $ \alpha (x) $ to be regularly varying at infinity,
\begin{equation}\label{reg}
\alpha(x)=x^{\theta} L(x) \ \  \textrm{with} \ \  \theta\in [0,1],
\end{equation}
  where $L(x)$ is a function slowly varying at infinity. 
Clearly, $L(t)\to 0$ as $t\to \infty$, if $\theta=1$. 
Lemma 4 of Karlin showed that function 
\[
L^{*}(t)\stackrel{def}{=}\int_0^{\infty}\frac{e^{-1/y}}{y}L(t y)dy\to 0 \ \ \textrm{as} \ \ t\to\infty
\]
is slow varying, too.  Let, for $k\geq 1$, 
\[
Y_{n,k}^{*} =R^{*}_{n,k}-{\mathbb E} R^{*}_{n,k}, 
\ \ \ \ Y_{n,k}=R_{n,k}-{\mathbb E} R_{n,k}, 
\]
\[
B_{n,k}^{*} ={\mathbb Var}R^*_{P(n),k}, \ \ \ \  B_{n,k}={\mathbb Var}R_{P(n),k}
\]
and let $R_n\stackrel{def}{=}R^{*}_{n,1}=\sum_{k\ge 1} R_{n,k}$ be the number of non-empty cells. 
Karlin has established a number of asymptotic properties of random variables $R_n$ as $n\to\infty$, including the Strong Law of Large Numbers (SLLN) and the asymptotic normality in the range $ \theta \in (0,1] $, and also the asymptotic normality of random vector  $(R_{n,1}, ..., R_{n,k})$, $k\ge 1$ when  $\theta\in(0,1)$.
The proof of normality was based on the following convergences: as $n\to\infty$
\begin{equation}\label{*}
 {\mathbb E}R_n-{\mathbb E}R_{P(n)}\to 0
\end{equation}
and under condition (\ref{reg}), for any fixed $c_0>0$, $\theta\in(0,1]$ and $c_1$
\begin{equation}\label{**}
 \sup_{|c|\le c_0}\frac{|{\mathbb E}R_{[n+c\sqrt{n}]}-{\mathbb E}R_n|}{\sqrt{ B_{n,1}^*}}\to 0,
\end{equation}
\begin{equation}\label{***}
 \frac{R_{[n+c_1\sqrt{n}]}-R_n}{\sqrt{ B_{n,1}^*}}\xrightarrow[]{\textrm{p}} 0.
\end{equation}

Dutko (1989) has proved the asymptotic normality of $R_n$ under a weaker assumption. Namely, he replaced regular condition   (\ref{reg}) by the following: 
\begin{equation}\label{nreg}
 B_{n,1}^* \to \infty \ \ \textrm{as} \ \ n\to\infty.
\end{equation}
For the regularly varying tails, condition (\ref{reg}) holds for any $ \theta \in (0,1]$ and may also hold for $ \theta = 0 $ in a particular case. 
Dutko did not assume condition (\ref{reg}) in his proofs of  (\ref{**}) and (\ref{***}).

Gnedin, Hansen and Pitman (2007) have studied sufficient conditions for (\ref{nreg}), found rate of convergence in (\ref{*})  and provided an overview on the topic. 

Hwang and Janson (2008) have proved local limit theorems for a finite and infinite number of cells.

Barbour and Gnedin  (2009) have proved asymptotic normality of random vector 
$ (R_{n, 1}, \ldots, R_{n, k}) $ for $ k \ge1 $ 
under the condition $B_{n,i}\to\infty$ as $n\to\infty$, for any $i=1,\dots,k$.
Note that it is sufficient to have $B_{n,k}\to\infty$ or ${\mathbb E}R_{P(n),k}\to\infty$ (see Lemma 5).  
They have obtained (in their Lemma 2.1) an upper bound for the total variation distance between vectors
 $ (R_{n, 1}, \ldots, R_{n, k})$ and $ (R_{P(n), 1}, \ldots, R_{P(n), k})$, and also showed that 
the covariance matrices converge if and only if condition (\ref{reg}) holds.  

Barbour (2009) has proved theorems on approximation of the number of cells with $k$ balls by translated Poisson distribution in the total variation distance.

Chebunin and Kovalevskii (2016) have proved the Functional Central Limit Theorem for random vector $(R^*_{n,1}, ..., R^*_{n,k})$,  
$\theta\in(0,1), \ k\ge1$. Their proof is based on the convergence 
\[
\sup_{0\le t\le 1}\frac{|R_{P(nt)}-R_{[nt]}|}{\sqrt{ B_{n,1}^*}}\xrightarrow[]{\textrm{p}} 0  \ \textrm{   as  $n\to\infty$.}
\]

Zakrevskaya and Kovalevskii (2001) have proposed an implicit estimator of parameter $ \theta $ based on $ R_n $ for one-parametric family 
and proved its consistency.

Chebunin (2014) has proposed explicit estimators of the parameter based on $R_n$ for a broader class of distributions 
and proved their consistency.

In this paper, we analyse accuracy of a.s. approximation of $R^*_{n,k}$ by $R^*_{P(n),k}$
when $n$ grows, for any fixed $k\ge 1$.

\begin{Theorem}
 Under condition (\ref{reg}), for any $k\ge1$ and $\theta\in[0,1]$,
\[
b_n (R^*_{n,k}-R^*_{P(n),k})\xrightarrow[n\to\infty]{\textrm{a.s.}} 0, \ \   \textrm{and} \ \
 \ b_n (R_{n,k}-R_{P(n),k})\xrightarrow[n\to\infty]{\textrm{a.s.}} 0.
\]
Here
\begin{equation}\label{b_n}
b_n=\left\{
\begin{array}{ll}
(n L^{*}(n)\ln\ln n)^{-\frac12}, & \theta=1, \ k=1; \\
(n L(n)\ln\ln n)^{-\frac12}, & \theta=1, \ k\ge 2; \\
o(\min\{n^{\frac12-\theta}(\ln\ln n L(n))^{-1},(\ln n)^{-1}\}), & \theta<1, \ k\ge 1.
\end{array}
\right.
\end{equation}
\end{Theorem}

Note that for $\theta=1$, sequence $b_n$ is what we could expect to appear
 in the Law of the Iterated Logarithm (LIL) for $R_{n,k}$.
For $\theta\in[1/2,1)$, sequence $b_n$ is better then the normalizing constant in CLT. For $\theta<1/2$,  
 coefficients $b_n=o((\ln n)^{-1})$ do not depend on $\theta$. 
As a corollary, we obtain asymptotic upper bounds for the absolute values of $Y^*_{n,k}$.
 
\begin{Corollary}
 Assume (\ref{reg}) to hold. 
If $\frac{{\mathbb E} R_{P(n),k_0}}{\ln n} \to \infty$ as $n\to\infty$, for some $k_0\ge1$,  then, for any  $k\le k_0$,
\[
{ \mathbb P} \left( \limsup\limits_{n\to\infty}\frac{|Y^*_{n,k}|}{ \sqrt{2 B^*_{n,k} \ln n}} \le 1\right)=1, \ \
{ \mathbb P} \left( \limsup\limits_{n\to\infty}\frac{|Y_{n,k}|}{ \sqrt{2 B_{n,k} \ln n}} \le 1\right)=1.
\] 
\end{Corollary}

Note that for $\theta\in(0,1]$ the assumptions of the corollary 1 are held for all $k_0 \ge 1$ (this follows from the 
asymptotics of $B^*_{n,k}$, see Lemma 1).

\begin{Remark}
As it follows from Lemma 1 in Gnedin, Hansen, Pitman (2007), ${\mathbb E } R_n-{\mathbb E } R_{P(n)}\to 0$, ${\mathbb E } R_{n,k}- {\mathbb E } R_{P(n),k}\to 0$ as $n\to\infty$. Then ${\mathbb E } R^*_{n,k}- {\mathbb E } R^*_{P(n),k}\to 0$ too, 
since $R^*_{n,k}=R_n - R_{n,1} - ... -R_{n,k-1}$.
To establish the LIL for non-random scheme of size $n=1,2,\dots$, it suffices to prove the LIL for the poissonized scheme, with normalising of order $o((\ln n)^{-1})$. We could not manage to prove the latter. 
However, we obtain a weaker result (corollary 1) which may be viewed as an analogue of the LIL for arrays of random variables 
(see Sung (1996) and Hoffmann, Miao, Li, Xu (2016) for further comments and background).
 \end{Remark}

The rest of the paper is organized as follows.
In Section 2-3 we formulate all the auxiliary results and prove Theorem 1 and Corollary 1. Appendix contains proofs of auxiliary results.

\section{Proof of Theorem 1}

Recall basic properties of Poisson process.
\begin{Proposition} 
Let $v_t$ be a positive function such that $v_t/\sqrt{t\ln\ln t}\to\infty$ as $t\to\infty$.
Then as $t\to\infty$
\[
\sup\limits_{ w_t\ge v_t}\frac{P(t+ w_t)-P(t)}{w_t}\to 1 \ \textrm{ a.s. }
\] 
\end{Proposition}
Clearly, 
\[
\frac{P(t+ w_t)-P(t)}{w_t}-1=\frac{P(t+ w_t)-t- w_t}{w_t}-\frac{P(t)-t}{w_t},
\]
\[
\frac{w_t}{\sqrt{(t+w_t)\ln\ln(t+w_t)}}\ge\min\left( \frac{w_t}{\sqrt{2t\ln\ln(2t)}},\frac{w_t}{\sqrt{2w_t\ln\ln(2w_t)}}\right)\xrightarrow[t\to\infty]{}\infty.
\]
 So by the LIL for  $P(t)$ we have 
\[
\sup\limits_{ w_t\ge v_t} \frac{P(t)-t}{w_t}\xrightarrow[t\to\infty]{\textrm{a.s.}}  0, \ \ \ \textrm{and} \ \   \ 
\sup\limits_{ w_t\ge v_t} \frac{P(t+w_t)-t-w_t}{w_t}\xrightarrow[t\to\infty]{\textrm{a.s.}} 0.
\]

We also need the following auxiliary results from
Karlin (1967, Theorem 1, formulas (23), (26) and (37)). As $t\to\infty$,
\[
{\mathbb E} R_{P (t)}\sim {\mathbb Var} R_{P (t)}\sim {\mathbb E} R_{P (t),1}\sim  {\mathbb Var} R_{P (t),1}\sim  t L^{*}(t) \ \ \textrm{if} \ \ \theta=1,
\]
 \[
{\mathbb E} R_{P (t)}\sim \Gamma(1-\theta) \alpha(t), \ {\mathbb Var} R_{P (t)}\sim \Gamma(1-\theta)(2^\theta -1) \alpha(t)\ \ \textrm{if} \ \ \theta\in(0,1),
\]
\[
{\mathbb E} R_{P (t)}\sim \alpha(t), \ \ {\mathbb Var} R_{P (t)}\sim \alpha(2t)-\alpha(t)\ \ \textrm{if} \ \ \theta=0,
\]
 \[
{\mathbb E} R_{P (t),k}\sim \theta\frac{\Gamma(k-\theta)}{k!} \alpha(t), \ \ \
{\mathbb Var} R_{P (t),k}\sim \frac{\theta}{k!}\left[\Gamma(k-\theta)-\frac{\Gamma(2k-\theta)}{2^{2k-\theta}k!}\right] \alpha(t)
\]
if either $\theta\in(0,1), \ k\ge1$ or $\theta=1, \ k\ge2$.

The proofs of the following lemmas may be found in Appendix.

\begin{Lemma}
For $k\ge2$ and as $t\to\infty$ 
\[
{\mathbb E} (R^{*}_{P(t),k})\sim\left\{
\begin{array}{ll}
\theta \sum\limits_{i=k}^\infty \frac{\Gamma(i-\theta)}{i!} \alpha(t), & \textrm{if} \  \ \theta\in(0,1]; \\
\alpha(t), & \textrm{if} \  \  \theta=0,  
\end{array}
\right.
\]
If $\theta\in(0,1]$, then also
\[
{\mathbb Var} (R^{*}_{P(t),k})\sim \left(
2^{\theta} \Gamma(2-\theta)-\frac{ \Gamma(k-\theta)}{ (k-1)!} 
-\theta \sum\limits_{s=0}^{k-1}\sum\limits_{m=0}^{k-1}\frac{{\mathbb I}\{ s+m\ge2 \} \Gamma(s+m-\theta)}{2^{s+m-\theta} s! m!}\right) \alpha(t).
\]
\end{Lemma}

\begin{Lemma}

Let $t_n=o(n)$ as $n\to\infty$. Then, for any $k\ge1$ and $\theta\in[0,1]$, there exists $n_0$ such that, for $n \ge n_0$,

\[ 
|{\mathbb E} R^*_{P (n+t_n),k}-{\mathbb E} R^*_{P (n),k}|\le \frac{2 |t_n|}{n} {\mathbb E} R^*_{P (n),k}.
\]
 \end{Lemma}

Let for $k\geq 1$ 
\[
Z_{n,k}^{*} =R^{*}_{P (n),k}-{\mathbb E} R^{*}_{P (n),k}, \ \ \ \ \ Z_{n,k} =R_{P (n),k}-{\mathbb E} R_{P (n),k}.
\]

For any sequence $t_n=o(n)$ as $n\to\infty$, 
we may introduce a positive sequence $a_n=a_n(k,\theta,t_n)$ satisfying the following constraints:
\[
a_n=\left\{
\begin{array}{ll}
o\left(\frac{\min \{(|t_n| L^*(n) )^{-\frac12} , 1 \}} {\ln n}\right), & \textrm{if} \ \ \theta=1, \ k=1; \\
o\left(\frac{\min \{(|t_n| n^{\theta-1} L(n) )^{-\frac12} , 1 \}}{\ln n}\right), & \ \textrm{if either $\theta\in[0,1)$ or $k\ge 2$}.  
\end{array}
\right.
\]

\begin{Lemma} For any $k\ge1$, and $\theta\in[0,1]$,
\[
a_n(Z^*_{n+t_n,k}-Z^*_{n,k})\xrightarrow[n\to\infty]{\textrm{a.s.}} 0, \ \   \textrm{and} \ \
  a_n (Z_{n+t_n,k}-Z_{n,k})\xrightarrow[n\to\infty]{\textrm{a.s.}} 0.
\] 
\end{Lemma}

\begin{proof}
Let for $n_2 \geq n_1 \geq 0$:
\[
R^{*}_{P(n_2),k}-R^{*}_{P(n_1),k}=\sum_{i=1}^{\infty} {\mathbb I}(P_i(n_2))\ge k, \ P_i(n_1)< k)
\stackrel{def}{=}\sum_{i=1}^{\infty} {\bf I}_{i}, \ \ \textrm{and} \ \ {\mathbb P}({\bf I}_{i})= {\bf P}_i.
\]

Since 
\begin{equation}\label{Bound}
\sum_{i=1}^{\infty} {\bf I}_{i}\le P(n_2)-P(n_1) \ \ \textrm{a.s.,}
\end{equation}
 and the variance of an indicator random variable is not bigger than its expectation,
we have 
\begin{equation}\label{V}
{\mathbb Var} \left(\sum_{i=1}^{\infty} {\bf I}_{i}\right)\le \sum_{i=1}^{\infty} {\bf P}_{i}\le 
{\mathbb E} (P(n_2)-P(n_1)) = n_2-n_1.
\end{equation}
  For any fixed $n_1, \ n_2$ and $C>0$, we have
\[
{\mathbb E} \exp\left\{ C \sum_{i=1}^{\infty} ({\bf I}_{i}- {\bf P}_{i})  \right\}
\]
\begin{equation}\label{E}
\le 
{\mathbb E} \exp\left\{ C (P(n_2)-P(n_1))  \right\}=\exp\{ (n_2-n_1)(e^C-1)\}<\infty.
\end{equation}
Since $C>0$, (\ref{Bound}) and (\ref{E}) imply uniform integrability of sequence
 \[
\left\{\exp\left( C \sum_{i=1}^{N} ({\bf I}_{i}- {\bf P}_{i})\right)\right\}_{N=1}^\infty. 
\]

By inequalities $e^x\le 1+x+\frac{x^2}2 e^{|x|}$ for all $x\in {\mathbb R}$ and $|{\bf I}_{i}-{\bf P}_i|\le 1$ a.s., we have,
 for any $\gamma_n>0$,
\[
{\mathbb E} e^{\gamma_n a_n ({\bf I}_{i}-{\bf P}_i)}\le 1+ \frac{(\gamma_n a_n)^2}{2 } e^{\gamma_n a_n} {\mathbb E} ({\bf I}_{i}-{\bf P}_i)^2\le \exp\left\{\frac{(\gamma_n a_n)^2}{2} e^{\gamma_n a_n} {\bf P}_i\right\}. 
\]

Since $\{ {\bf I}_i\}_{i=1}^\infty$ are mutually independent, we have, for any  $N\ge 1$, 
\[
{\mathbb E} \exp\left\{\gamma_n a_n \sum_{i=1}^N ({\bf I}_{i}-{\bf P}_i)\right\}=
\prod_{i=1}^N {\mathbb E} \exp\left\{\gamma_n a_n ({\bf I}_{i}-{\bf P}_i)\right\} 
\]
\begin{equation}\label{ri}
\le\exp\left\{\frac{(\gamma_n a_n)^2}{2} e^{\gamma_n a_n} \sum_{i=1}^N{\bf P}_i\right\}.
\end{equation}

  Since (\ref{V}), by the Chebyshev inequality, we have $\sum_{i=N}^{\infty} ({\bf I}_{i}- {\bf P}_{i})\to 0$ in probability as  $N\to\infty$, for any fixed $n_1, \ n_2$.
Since $e^{Cx}$ is a continuous function,  
$\exp\left\{C \sum_{i=1}^{N} ({\bf I}_{i}- {\bf P}_{i})\right\}-\exp\left\{C \sum_{i=1}^{\infty} ({\bf I}_{i}- {\bf P}_{i})\right\}\to 0$ in probability as  $N\to\infty$.
Letting $N\to\infty$ in (\ref{ri}), we get 
 \[
{\mathbb E}\exp \left\{ \gamma_n a_n \sum_{i=1}^{\infty} ({\bf I}_{i}-{\bf P}_i)\right\} \leq 
\exp\left\{\frac{(\gamma_n a_n)^2 e^{\gamma_n  a_n}}{2 }  ({\mathbb E}R^{*}_{P(n_2),k}-{\mathbb E}R^{*}_{P(n_1),k})\right\}. 
\]

From Lemma~2, Lemma~1 and the definition of  $a_n$, we have, as $n\to\infty$,
\[
a^2_n|{\mathbb E} R^*_{P (n+t_n),k} -{\mathbb E} R^*_{P (n),k}| = o((\ln n)^{-2}). 
\]
   Let $\gamma_n=3 \ln n/ \eta$, then
\[
{\mathbb P}(a_n \left|Z^*_{n+t_n,k}-Z^*_{n,k}\right|\ge \eta ) 
= {\mathbb P}( a_n(Z^*_{n+t_n,k}-Z^*_{n,k})\ge \eta)+{\mathbb P}(a_n(Z^*_{n,k}- Z^*_{n+t_n,k})\ge \eta) 
\]
\[
\le 2 \exp\left\{\frac{(\gamma_n a_n)^2 e^{\gamma_n a_n}}{2 } |{\mathbb E} R^*_{P (n+t_n),k} -{\mathbb E} R^*_{P (n),k}| - t\eta \right\}
= 2 \exp\left\{ o(1) - 3\ln n \right\}\le \frac2{n^2}. 
 \]
By the Markov inequality, for any pair $\varepsilon>0$,  $\eta>0$, there is integer $n_0$ such that, for $n\ge n_0$,
\[
{\mathbb P}(\sup_{n\ge n_0} a_n \left|Z^*_{n+t_n,k}-Z^*_{n,k}\right| \ge \eta)
\le \sum_{n=n_0}^\infty {\mathbb P}(a_n \left|Z^*_{n+t_n,k}-Z^*_{n,k}\right|\ge \eta) 
 \le 
\sum_{n=n_0}^\infty \frac{2 }{n^2 }\le
\varepsilon.
\]

The second assertion of the lemma follows directly.
\end{proof}

Let
\[
t'_n=\left\{
\begin{array}{ll}
\sqrt{n\ln\ln n} (L^*(n))^{-1/4}, & \textrm{if} \ \  \theta=1, \ k=1; \\
\sqrt{n\ln\ln n} (L(n))^{-1/4}, & \textrm{if} \ \ \theta=1, \ k\ge 2; \\
\sqrt{n} \ln\ln n, & \textrm{if} \ \ \theta<1. 
\end{array}
\right.
\]
Introduce the satisfies by following constraints sequence $a'_n=a_n(k,\theta,t'_n)$:
\begin{equation}\label{a'}
a'_n=\left\{
\begin{array}{ll}
o( (n\ln \ln n)^{-\frac14} (L^*(n))^{-\frac38} (\ln n)^{-1}), & \textrm{if} \ \  \theta=1, \ k=1; \\
o( (n\ln \ln n)^{-\frac14} (L(n))^{-\frac38} (\ln n)^{-1}), & \textrm{if} \ \  \theta=1, \ k\ge 2; \\
o(\min \{ n^{\frac{1-2\theta}4} (L(n) \ln \ln n )^{-\frac12} , 1 \} (\ln n)^{-1}), & \textrm{if} \ \  \theta<1. 
\end{array}
\right.
\end{equation}

\it
Proof of Theorem 1.
\rm
\\
Clearly, the sequence $b_n$  in  (\ref{b_n})  satisfies  conditions (\ref{a'}).
By Proposition 1 and the LIL for $P(t)$, we have 
\[
\frac{P(n)-n}{P(n+t'_n)-P(n)}\xrightarrow[n\to\infty]{\textrm{a.s.}} 0 ,\ \ \  \textrm{and} \ \ \  \frac{P(n)-n}{P(n-t'_n)-P(n)}\xrightarrow[n\to\infty]{\textrm{a.s.}} 0.
\] 

By monotonicity of $P(t)$, for any $\varepsilon\in(0,1)$, there exists an integer 
$n_0$ such that
\begin{equation}\label{pois1}
{\mathbb P}( \forall n\ge n_0 \  \ \exists \delta_n \ : \  |\delta_n|\le 1,  \ P(n+\delta_nt'_n)=n) \stackrel{def}{=} {\mathbb P} (A(n_0))\ge1-\varepsilon/2.
\end{equation}

By Lemma~1 and Lemma~2, and using direct calculations, we may conclude that 
\begin{equation}\label{pois2}
b_n({\mathbb E} R^*_{P(n \pm t'_n),k}-{\mathbb E}R^*_{P(n),k})\to 0 \ \ \ \textrm{as $n\to\infty$.}
\end{equation}

Since $R^*_{n,k}=R^*_{P(n+\delta_n t'_n),k}$ a.s., if $P(n+\delta_n t'_n)=n$. Then given the event $A(n_0)$ occurs, we have, for any $n\ge n_0$ with probability one
$$
|R^*_{n,k}-R^*_{P(n),k}|\le \sup_{|\delta|\le 1}|R^*_{P(n+\delta t'_n),k}-R^*_{P(n),k}|
$$
$$
=\max(R^*_{P(n+t'_n),k}-R^*_{P(n),k},R^*_{P(n),k}-R^*_{P(n-t'_n),k}),
$$
due to monotonicity $P(t)$ in $t$ and $R^*_{n,k}$ in $n$.

So, by (\ref{pois1}), (\ref{pois2}) and Lemma~3, for any pair $\varepsilon>0$,  $\eta>0$, 
there exists an integer $n_0$ such that for $n\ge n_0$
\[
{\mathbb P}\left(\sup_{n\ge n_0} b_n |R^*_{n,k}-R^*_{P(n),k}|\ge \eta\right)
\le {\mathbb P}\left(\sup_{n\ge n_0} b_n|R^*_{n,k}-R^*_{P(n),k}|\ge \eta, A(n_0)\right)+\frac\varepsilon 2
\]
\[
\le {\mathbb P}\left(\sup_{n\ge n_0} b_n |Z^*_{n-t'_n,k}-Z^*_{n,k}|\ge \frac\eta2 \right)
+ {\mathbb P}\left(\sup_{n\ge n_0} b_n|Z^*_{n+t'_n,k}-Z^*_{n,k}|\ge \frac\eta2\right)+\frac \varepsilon 2  \le \varepsilon.
\]

The second assertion of the Theorem 1 follows directly.\\
{ \it Theorem 1 is proved.}

\section{Proof of Corollary 1}

 For any $n\ge1$, let random variables  $\{\xi_{n,i}\}_{i\ge1}$  be mutually independent with  
${\mathbb E}\xi_{n,i}=0$ for $i\ge1$. Let 
$S_{n,N}=\sum\limits_{i=1}^N \xi_{n,i}$, $s_{n,N}^2=\sum\limits_{i=1}^N {\mathbb E}\xi^2_{n,i}$
for $N\ge 1$, and $S_n=S_{n,\infty}, \ s^2_n=s^2_{n,\infty}>0$. 
Analogously to Lemma~1 in Sung (1996) we prove the following lemma for any dependence between strings. 
The proofs of the following lemmas may be found in Appendix.

\begin{Lemma}
Let $c_n$ be a sequence of positive constants such that  $c_n\to 0$ as $n\to\infty$. Let 
$
|\xi_{n,i}|\le c_n s_n/\sqrt{\ln n}
$
a.s. for all $n, i \ge 1$. Let the sequence $\{ e^{C S_{n,N}}\}_{N=1}^\infty$ be uniformly integrable for any fixed $n\ge1$ and $C>0$, and $s^2_n<\infty$. 
Then
\[
{\mathbb P}\left( \limsup\limits_{n\to\infty}\frac{S_n}{\sqrt{2 s_n^2 \ln n }}\le1 \right)=1.
\]
\end{Lemma}

\begin{Lemma}

Let $d_n$ be a sequence of positive constants such that, $d_{[c n]}/d_n>\varepsilon(c)>0$ for any $c>0, \ n\ge n_0$.
 Then the conditions $\min\limits_{1\le k\le k_0}(B^*_{n,k},B_{n,k})/d_n \to \infty$ and ${\mathbb E} R_{P(n),k_0}/d_n\to\infty$ 
as $n\to\infty,$  are equivalent.
\end{Lemma}

\it
Proof of Corollary 1.
\rm

By Theorem 1, it is enough to prove similar assertions for $Z^*_{n,k}$ and  $Z_{n,k}$ (with normalization $o((\ln n)^{-1})$).
We use Lemma~4. Let, for $k\in \{1, ... , k_0\}$, 
\[
 \xi^*_{n,i}=\pm({\mathbb I}(J_i(P(n))\ge k) - {\mathbb P}(J_i(P(n))\ge k)), 
\]
\[
\xi_{n,i}=\pm({\mathbb I}(J_i(P(n))= k) - {\mathbb P}(J_i(P(n))= k)).
\]
Then, for $n,i\ge 1$, 
\[
(s^*_n)^2=B^*_{n,k}, \ \  s_n^2=B_{n,k}, \ \  \textrm{and} \ \ |\xi^*_{n,i}|\le1, \ \ |\xi_{n,i}|\le1. 
\]   
As $R_{P(n),k}\le R^*_{P(n),k}\le P(n)$ a.s., and the variance of an indicator random variable is not bigger than its expectation,
we have 
\[
(s^*_n)^2\le {\mathbb E}R^*_{P(n),k}\le {\mathbb E} P(n) = n.
\]
 Similarly, we get $s^2_n\le n$. For any fixed $n$ and $C>0$, we have
\[
{\mathbb E} e^{C S^*_{n,N}}\le {\mathbb E} e^{C (P(n)+n)}=\exp\{ n(e^C+C-1)\}<\infty.
\]
Then sequence $\{ e^{C S^*_{n,N}}\}_{N=1}^\infty$ is uniformly integrable. The same holds for the sequence  $\{ e^{C S_{n,N}}\}_{N=1}^\infty$.

Let 
\[
c_n^*= \left(\frac{\ln n}{ B^*_{n,k} }\right)^{1/4} \ \  \textrm{ and } \ \  c_n=\left(\frac{\ln n}{ B_{n,k} }\right)^{1/4},
\]
 then, by Lemma~5, 
\[
|\xi^*_{n,i}|\le \frac{c_n^* s^*_n}{\sqrt{\ln n}} \ \ \textrm{ and } \ \  |\xi_{n,i}|\le \frac{c_n s_n}{\sqrt{\ln n}} \ \ \textrm{ a.s. }
\] 
Then the required result follows from Lemma 4. 
\\
{ \it Corollary 1 is proved.}

\section*{Appendix}

\it
Proof of Lemma~1.
\rm

Since $R^{*}_{P(t),k}=R_{P(t)} - \sum\limits_{i=1}^{k-1} R_{P(t),i}$ where the sum is finite and the asymptotics for each term is known, the asymptotics for ${\mathbb E}R^{*}_{P(t),k}$  follows from formula  (23) in Karlin (1967) if $\theta>0$.
Let us analyse the asymptotic behaviour of ${\mathbb E}R^{*}_{P(t),k}$ as $t\to\infty$, when $\theta=0$. Note that, for $i\ge1, \ t>0$,
$$
{\mathbb E}R_{P(t),i}=\sum_{j=1}^\infty \frac{(tp_j)^i}{i!}e^{-tp_j}=
\frac1{i!} \int_0^\infty \frac{t^i}{x^i}e^{-t/x}d\alpha(x)
$$
$$
 = 
\frac1{i!} \int_0^\infty \left(\frac{i t^i}{x^{i+1}}- \frac{ t^{i+1}}{x^{i+2}} \right)e^{-t/x}\alpha(x) dx = 
\frac1{i!} \int_0^\infty \left(\frac{i }{y^{i+1}}- \frac{ 1}{y^{i+2}} \right)e^{-1/y}\alpha(ty) dy.
$$

So $\frac{{\mathbb E}R_{P(t),i}}{\alpha(t)}\xrightarrow[t\to\infty]{} 0$ if $\theta=0$.
Since ${\mathbb E}R_{P(t)}\sim \alpha(t)$ as $t\to\infty$, we obtain the required result. 
The variance of $R^{*}_{P(t),k}$ for $\theta\in(0,1]$ may be found by
\[
{\mathbb Var} (R^{*}_{P(t),k}) 
=
\sum_{i=1}^{\infty} {\mathbb P} (J_i(P(t))\geq k)(1-{\mathbb P} (J_i(P(t))\geq k))
\]
\[
=
\sum_{i=1}^{\infty} {\mathbb P} (P(tp_i)< k)(1-{\mathbb P} (P(tp_i)< k))
=
\sum_{i=1}^{\infty} \sum_{s=0}^{k-1} \frac{(t p_i)^s}{s!}e^{-t p_i} \left(1- \sum_{m=0}^{k-1} \frac{(t p_i)^m}{m!}e^{-t p_i}\right)
\]
\[
=
\int_0^{\infty} \sum_{s=0}^{k-1} \frac{t^s x^{-s}}{s!}e^{- t/x} 
\left(1-\sum_{m=0}^{k-1}\frac{t^m x^{-m}}{m!}e^{- t/x} \right) d\alpha(x).
\]

We use integration by parts and decomposition into two integrals:
\[
{\mathbb Var} (R^{*}_{P(t),k}) =
\int_0^{\infty} \sum_{s=0}^{k-1} \frac{t^s }{s!} (s x^{-s-1}-t x^{-s-2})
e^{-t/x} 
\alpha(x) dx
\]
\[
-
\int_0^{\infty} \sum_{s=0}^{k-1}\sum_{m=0}^{k-1} \frac{t^{s+m} }{s! m!} ((s+m) x^{-s-m-1}-2t x^{-s-m-2})
e^{-2t/x} 
\alpha(x) dx.
\]

Let substitute  $y=x/t$ in the first integral, and $y=x/(2t)$ in the second integral:
\[
{\mathbb Var} (R^{*}_{P(t),k}) =
\sum_{s=0}^{k-1} \frac{1 }{s!}  \int_0^{\infty} (s y^{-s-1}- y^{-s-2})
e^{-1/y} \alpha(t y) dy
\]
\[
-
\sum_{s=0}^{k-1}\sum_{m=0}^{k-1} \frac{2^{-s-m} }{s! m!} 
\int_0^{\infty} ((s+m) y^{-s-m-1}-y^{-s-m-2}) e^{-1/y} \alpha(2ty) dy
\]
\[
=-\frac1{(k-1)!}\int_0^{\infty} y^{-k-1} e^{-1/y} \alpha(t y) dy 
+\int_0^{\infty} y^{-3} e^{-1/y} \alpha(2t y) dy
\]
\[
-\sum_{s=0}^{k-1}\sum_{m=0}^{k-1}\frac{{\mathbb I}\{ s+m\ge2 \}}{2^{s+m} s! m!} \int_0^{\infty} \left((s+m) y^{-s-m-1}-y^{-s-m-2}\right ) e^{-1/y} \alpha(2ty) dy.
\]

Since $\alpha(x)=x^{\theta}L(x)$, for any integer $r\geq 0$ and as $t\to \infty$,
\[
\int_0^{\infty} y^{-r-2}e^{-1/y} 
\alpha(ty) dy \sim 
\alpha(t) \int_0^{\infty} y^{\theta-r-2}e^{-1/y} 
dy =\alpha(t)\Gamma(r+1-\theta).
\]

Note that, for any integer $r\geq 2$,
\[
 \int_0^{\infty}(r  y^{\theta-r-1}- y^{\theta-r-2})e^{-1/y} dy=\theta \Gamma(r-\theta).
\]
{ \it Lemma~1 is proved.}
\\
\\
\it
Proof of Lemma~2.
\rm
\[
|{\mathbb E} R^*_{P (n+ t_n),k} - {\mathbb E} R^*_{P (n),k}|=
{\mathbb E} R^*_{P (n),k} \left|\frac{{\mathbb E} R^*_{P (n+ t_n),k}}{{\mathbb E} R^*_{P (n),k}} - 1\right|.
\]
By Lemma 1 and Karamata representation for slowly varying functions (Theorem 1.3.1,  in Bingham, Goldie, Teugels (1989)), 
we have that, for $\theta\in[0,1]$, and  $n\to\infty$
\[
\frac{{\mathbb E} R^*_{P (n+ t_n),k}}{{\mathbb E} R^*_{P (n),k}} \sim  \left(1+\frac{t_n}{n}\right)^{\theta} e^{ o\left(\ln\left(1+\frac{t_n}{n}\right)\right)}=
 \left(1+\frac{t_n}{n}\right)^{\theta+o(1)}
=1+ \frac{t_n}{n}(\theta+o(1)).
\]
{ \it Lemma~2 is proved.}
\\
\\
\it
Proof of Lemma~4.
\rm
\\
Let  $l_n=\sqrt{2 s_n^2 \ln n }$ for $ n\ge 1$. Then, by Borel-Cantelli lemma, it suffices to show that, for any $ \varepsilon> 0 $, 
\[
\sum\limits_{n=1}^\infty {\mathbb P}(S_n/l_n>1+\varepsilon)<\infty.
\] 

From inequalities $e^x\le 1+x+\frac{x^2}2 e^{|x|}$ for all $x\in {\mathbb R}$ and $|\xi_{n,i}|/l_n\le c_n/(\sqrt{2} \ln n)$ a.s., we have, for $t>0$,
\[
{\mathbb E} e^{t \xi_{n,i}/l_n}\le 1+ \frac{t^2}{2 l_n^2} e^{t c_n/(\sqrt{2} \ln n)} {\mathbb E} \xi^2_{n,i}\le
 \exp\left\{\frac{t^2}{2 l_n^2} e^{t c_n/(\sqrt{2} \ln n)} {\mathbb E} \xi^2_{n,i}\right\}. 
\]

Since $\{ {\bf \xi_{n,i}}_i\}_{i=1}^\infty$ are mutually independent, we have, for any  $N\ge 1$, 
\begin{equation}\label{RI}
{\mathbb E} e^{t S_{n,N}/b_n}=\prod_{i=1}^N {\mathbb E} e^{t \xi_{n,i}/b_n} \le \exp\left\{\frac{t^2 s^2_{n,N}}{2 b_n^2} e^{t c_n/(\sqrt{2} \ln n)}\right\}. 
\end{equation}

 Let $t=2(1+\varepsilon)\ln n$ and $C=t/b_n$ for any fixed $n\ge1$. Since, we apply the $s^2_n<\infty$ then, by the Chebyshev inequality, we have, for any $\varepsilon>0$,
\[
{\mathbb P} (|S_{n,N}-S_n|>\varepsilon)\le \sum_{i=N+1}^\infty {\mathbb E} \xi^2_{n,N}/\varepsilon^2\to 0,
\]
that is, $S_{n,N}-S_n\to 0$ in probability as  $N\to\infty$.
It follows from the condition of the lemma that sequence $\{ e^{C S_{n,N}}\}_{N=1}^\infty$ is uniformly integrable and 
$e^{C S_{n,N}}-e^{C S_n}\xrightarrow[N\to\infty]{\textrm{p}} 0$  (as $e^{Cx}$ is a continuous function).
Letting $N\to\infty$ in (\ref{RI}), we get 
 \[
\lim_{N\to\infty} {\mathbb E} \exp\{C S_{n,N}\}={\mathbb E} \exp\{C S_n\}
\le \lim_{N\to\infty} \exp\left\{\frac{C^2 s^2_{n,N}}{2} e^{\sqrt{2}(1+\varepsilon)c_n}\right\}
\]
\[
= \exp\left\{\frac{C^2 s^2_n}{2} e^{\sqrt{2}(1+\varepsilon)c_n}\right\}=\exp\left\{\frac{t^2 }{4 \ln n} e^{\sqrt{2}(1+\varepsilon)c_n}\right\}.
\]

Since $c_n\to0$ as $n\to\infty$, for any $\varepsilon>0$, there exists $n_0\ge 1$ such that 
\[
(1+\varepsilon)^2(2-e^{\sqrt{2}(1+\varepsilon)c_n})>1+\varepsilon
\]
 as $n\ge n_0$. Then, by the Markov inequality, as $n\ge n_0$,
\[
{\mathbb P}(S_n/b_n>1+\varepsilon)\le e^{-t(1+\varepsilon)}{\mathbb E} \exp\{C S_n\}
\le \exp\left\{ -t(1+\varepsilon)+\frac{t^2}{4 \ln n} e^{\sqrt{2}(1+\varepsilon)c_n}\right\}
\]
\[
=\exp\{ - (1+\varepsilon)^2(2- e^{\sqrt{2}(1+\varepsilon)c_n}) \ln n \}\le1/n^{1+\varepsilon}.
\]
\\
{ \it Lemma~4 is proved.}
\\
\\
\it
Proof of Lemma~5.
\rm
\\
Note that 
\[
B^*_{n,k}=\sum_{i=1}^\infty {\mathbb P} (P(np_i)< k)(1-{\mathbb P} (P(np_i)< k))
\]
\[
\ge
\sum_{i=1}^\infty {\mathbb P} (P(np_i)=0) {\mathbb P} (P(np_i)= k)=\sum_{i=1}^\infty \frac{(n p_i)^k}{k!}e^{-2 n p_i}
=\frac1{2^k}{\mathbb E} R_{P(2n),k}.
\]
From  Barbour and Gnedin (2009, formulas (4.1), (4.2), (4.4)), there exist positive constants $c_{k}$ and $C_{k_0,k}$
 such that
${\mathbb E} R_{P(n),k}>B_{n,k}>c_{k}{\mathbb E} R_{P(n),k}$ and, for all $k<k_0$, the inequality 
${\mathbb E} R_{P(n),k}\ge C_{k_0,k} {\mathbb E} R_{P(2n),k_0}$ holds. 
From Proposition~3.2  in Ben-Hamou, Boucheron, and Ohannessian (2017), we have
$B^*_{n,k_0}\le k_0\cdot {\mathbb E} R_{P(n),k_0}$. Then
\[
{\mathbb E} R_{P(n),k_0}>B_{n,k_0}\ge \min\limits_{1\le k\le k_0}(B^*_{n,k},B_{n,k})\ge
 \min\limits_{1\le k\le k_0}(\frac1{2^k}{\mathbb E} R_{P(2n),k},c_{k}{\mathbb E} R_{P(n),k})
\]
\[
\ge
 \min\limits_{1\le k\le k_0}(\frac{C_{k_0,k}}{2^k}{\mathbb E} R_{P(4n),k_0},c_{k}C_{k_0,k}{\mathbb E} R_{P(2n),k_0}).
\]

Since $d_{[cn]}/d_n\ge \varepsilon(c)$ for $n\ge n_0$,
\[
\frac{{\mathbb E} R_{P([cn]),k_0}}{d_n}=\frac{{\mathbb E} R_{P([cn]),k_0}}{d_{[cn]}}\cdot \frac{d_{[cn]}}{d_n}\xrightarrow[n\to\infty]{}\infty \ 
\Leftrightarrow \ \frac{{\mathbb E} R_{P(n),k_0}}{d_n}\xrightarrow[n\to\infty]{}\infty.
\]
\\
{ \it Lemma~5 is proved.}
\\

{\bf Acknowledgements}
The research was supported by RFBR grant 17-01-00683. 
The author would like to thank Sergey Foss and Artyom Kovalevskii  
for their constant attention to the work.
\\
\\
\bigskip
\footnotesize

{\sc Barbour, A. D.}, 2009. Univariate approximations in the infinite occupancy
scheme. Alea 6, 415--433.

{\sc Barbour, A. D.,  Gnedin, A. V.}, 2009.
Small counts in the infinite occupancy scheme.
Electronic Journal of Probability, 
Vol. 14, Paper no. 13, 365--384.

{\sc Ben-Hamou, A., Boucheron, S., Ohannessian, M. I.}, 2017.
Concentration inequalities in the infinite urn scheme for occupancy counts and the missing mass, with applications.
Bernoulli, V. 23, 249--287.

{\sc Bingham, N. H., Goldie, C. M., Teugels, J. L.}, 1989. Regular variation, Cambridge University Press.

{\sc Chebunin, M. G.}, 2014. Estimation of parameters of probabilistic models which is based on the number of different elements 
in a sample. Sib. Zh. Ind. Mat., 17:3, 135--147 (in Russian).

{\sc Chebunin, M., Kovalevskii, A.}, 2016. 
Functional central limit theorems for certain statistics in an infinite urn scheme. 
Statistics and Probability Letters, V. 119, 344--348.

{\sc Dutko, M.}, 1989.
Central limit theorems for infinite urn models, Ann. Probab. 17,
1255--1263.

{\sc Gnedin, A., Hansen, B., Pitman, J.}, 2007.
Notes on the occupancy problem with
infinitely many boxes: general
asymptotics and power laws.
Probability Surveys,
Vol. 4, 146--171.

{\sc Hoffmann, J., Miao, Y., Li, X. C., Xu, S. F.}, 2016.  Kolmogorov Type Law of the Logarithm for Arrays. 
Journal of Theoretical Probability, Volume 29, Issue 1, 32--47.

{\sc Hwang, H.-K., Janson, S.}, 2008. Local Limit Theorems for Finite and Infinite 
Urn Models. The Annals of Probability, Vol. 36, No. 3,  992--1022.

{\sc Karlin, S.}, 1967. Central Limit Theorems for Certain Infinite Urn Schemes. 
Jounal of Mathematics and Mechanics, Vol. 17, No. 4,  373--401.

{\sc Sung, S. H.}, 1996.  An analogue of Kolmogorov's law of the iterated logarithm for arrays. Bull. Aust. Math.
Soc., Vol. 54, No. 2, 177--182.

{\sc Zakrevskaya, N. S.,  Kovalevskii, A. P.}, 2001. One-parameter probabilistic models of text statistics. 
Sib. Zh. Ind. Mat., 4:2, 142--153 (in Russian).

\end{document}